\documentclass{amsart}
\usepackage{amsmath}
\usepackage{amssymb}
\usepackage{graphicx}
\usepackage{amscd}

\newtheorem{theorem}{Theorem}
\theoremstyle{plain}

\newtheorem{corollary}{Corollary}

\newtheorem{definition}{Definition}

\newtheorem{lemma}{Lemma}
\newtheorem{notation}{Notation}

\newtheorem{proposition}{Proposition}
\newtheorem{remark}{Remark}

\numberwithin{equation}{section}

\input{tcilatex}

\begin{document}
\title[Maurer-Cartan equation]{Maurer-Cartan equation in the DGLA of graded
derivations }
\dedicatory{In the memory of Pierre Dolbeault}
\author{Paolo de Bartolomeis}
\address{Universit\`{a} degli Studi di Firenze\\
Dipartimento di Matematica e Informatica "U. Dini"\\
Viale Morgagni 67/A I-50134 \\
Firenze, Italia}
\author{Andrei Iordan}
\address{Sorbonne Universit\'{e}\\
Institut de Math\'{e}matiques de Jussieu-Paris Rive Gauche\\
UMR 7586 du CNRS, case 247 \\
4 Place Jussieu \\
75252 Paris Cedex 05\\
France}
\email{andrei.iordan@imj-prg.fr}
\date{September, 17, 2018}
\dedicatory{In the memory of Pierre Dolbeault}

\begin{abstract}
Let $M$ be a smooth manifold and $\Phi $ a differential $1$-form on $M$ with
values in the tangent bundle $TM$. We construct canonical solutions $e_{\Phi
}$ of Maurer-Cartan equation in the DGLA of graded derivations $\mathcal{D}%
^{\ast }\left( M\right) $ of differential forms on $M$ by means of
deformations of the $d$ operator depending on $\Phi $. This yields to a
classification of the canonical solutions of the Maurer-Cartan equation
according to their type: $e_{\Phi }$ is of finite type $r$ if there exists $%
r\in \mathbb{N}$ such that $\Phi ^{r}\left[ \Phi ,\Phi \right] _{\mathcal{FN}%
}=0$ and $r$ is minimal with this property, where $\left[ \cdot ,\cdot %
\right] _{\mathcal{FN}}$ is the Fr\"{o}licher-Nijenhuis bracket. A
distribution $\xi \subset TM$ \ of codimension $k\geqslant 1$ is integrable
if and only if the canonical solution $e_{\Phi }$ associated to the
endomorphism $\Phi $ of $TM$ which is trivial on $\xi $ and equal to the
identity on a complement of $\xi $ in $TM$ is of finite type $\leqslant $ $1$%
, respectively of finite type $0$ if $k=1$.
\end{abstract}

\maketitle

\section{Introduction}

In \cite{Kodaira61}, one of the last papers of their seminal cycle of works
on deformations of differentiable and complex structures, K. Kodaira and D.
C. Spencer studied the deformations of multifoliate structures. A $\mathcal{P%
}$-multifoliate structure on an orientable manifold $X$ of dimension $n$ is
an atlas $\left( U_{i},\left( x_{i}^{\alpha }\right) _{\alpha =1,\cdot \cdot
\cdot ,n}\right) $ such that the changes of coordinates verify%
\begin{equation*}
\frac{\partial x_{i}^{\alpha }}{\partial x_{k}^{\beta }}=0\ for\ \beta
\nsucceq \alpha ,
\end{equation*}%
where $\left( \mathcal{P},\geqq \right) $ is a finite partially ordered set, 
$\left\{ \alpha \right\} $ the set of integers $\alpha =1,2,\cdot \cdot
\cdot ,n$ such there is given a map $\left\{ \alpha \right\} \mapsto \left[
\alpha \right] $ of $\alpha $ onto $\mathcal{P}$ and the order relation $%
"\succapprox "$ is defined by $\alpha >\beta $ if and only if $\left[ \alpha %
\right] >\left[ \beta \right] $, $\alpha \thicksim \beta $ if and only if $%
\left[ \alpha \right] =\left[ \beta \right] $. An usual foliation is the
particular case when $\mathcal{P}=\left\{ a,b\right\} $, $a>b$.

They defined a DGLA structure $\left( \mathcal{D}^{\ast }\left( M\right)
,\daleth ,\left[ \cdot ,\cdot \right] \right) $ on the graded algebra of
graded derivations introduced by Fr\"{o}licher and Nijenhuis in \cite%
{Frolicher56} and the deformations of the multifoliate structures are
related to the solutions of the Maurer-Cartan equation in this algebra. This
was done in the spirit of \cite{Nijenhuis66}, where A. Nijenhuis and R.\ W.
Richardson adapted a theory initiated by M. Gerstenhaber \cite%
{Gerstenhaber64} and proved the connection between the deformations of
complex analytic structures and the theory of differential graded Lie
algebras (DGLA).

In the paper \cite{PdbAI15}, the authors elaborated a theory of deformations
of integrable distributions of codimension $1$ in smooth manifolds. Our
approach was different of K. Kodaira and D. C. Spencer's in \cite{Kodaira61}
(see remark 14 of \cite{PdbAI15} for a discussion). We considered in \cite%
{PdbAI15} only deformations of codimension $1$ foliations, the DGLA algebra $%
\left( \mathcal{Z}^{\ast }\left( L\right) ,\delta ,\left\{ \cdot ,\cdot
\right\} \right) $ associated to a codimension $1$ foliation on a
co-oriented manifold $L$ being a subalgebra of the the algebra $\left(
\Lambda ^{\ast }\left( L\right) ,\delta ,\left\{ \cdot ,\cdot \right\}
\right) $ of differential forms on $L$.\ Its definition depends on the
choice of a DGLA defining couple $\left( \gamma ,X\right) $, where $\gamma $
is a $1$-differential form on $L$ and $X$ is a vector field on $L$ such that 
$\gamma \left( X\right) =1$, but the cohomology classes of the underlying
differential vector space structure do not depend on its choice. The
deformations are given by forms in $\mathcal{Z}^{1}\left( L\right) $
verifying the Maurer-Cartan equation and the moduli space takes in account
the diffeomorphic deformations. The infinitesimal deformations along curves
are subsets of of the first cohomology group of the DGLA $\left( \mathcal{Z}%
^{\ast }\left( L\right) ,\delta ,\left\{ \cdot ,\cdot \right\} \right) $.

This theory was adapted to the study of the deformations of Levi-flat
hypersurfaces in complex manifolds: we parametrized the Levi-flat
hypersurfaces near a Levi-flat hypersurface in a complex manifold and we
obtained a second order elliptic partial differential equation for an
infinitesimal Levi-flat deformation.

In this paper we consider the graded algebra of graded derivations defined
by Fr\"{o}licher and Nijenhuis in \cite{Frolicher56} with the DGLA structure
defined by K. Kodaira and D. C. Spencer in \cite{Kodaira61}. We construct
canonical solutions of the Maurer-Cartan equation in this algebra by means
of deformations of the $d$-operator depending on a vector valued
differential $1$-form $\Phi $ and we give a classification of these
solutions depending on their type. A canonical solution of the Maurer-Cartan
equation associated to an endomorphism $\Phi $ is of finite type $r$ if
there exists $r\in \mathbb{N}$ such that $\Phi ^{r}\left[ \Phi ,\Phi \right]
_{\mathcal{FN}}=0$ and $r$ is minimal with this property, where $\left[
\cdot ,\cdot \right] _{\mathcal{FN}}$ is the Fr\"{o}licher-Nijenhuis
bracket. We show that a distribution $\xi $ of codimension $k$ on a smooth
manifold is integrable if and only if the canonical solution of the
Maurer-Cartan equation associated to the endomorphism of the tangent space
which is the trivial extension of the $k$-identity on a complement of $\xi $
in $TM$ is of finite type $\leqslant 1$. If $\xi $ is a distribution of
dimension $s$ such that there exists an integrable distribution $\xi ^{\ast
} $ of dimension $d$ generated by $\xi $, we show that there exists locally
an endomorphism $\Phi $ associated to $\xi $ such that the canonical
solution of the Maurer-Cartan equation associated to $\Phi $ is of finite
type less than $r=\min \left\{ m\in \mathbb{N}:\ m\geqslant \frac{d}{s}%
\right\} $.

In the case of integrable distributions of codimension $1$, we study also
the infinitesimal deformations of the canonical solutions of the
Maurer-Cartan equation in the algebra of graded derivations by means of the
theory of deformations developped in \cite{PdbAI15}.

\section{The $DGLA$ of graded derivations}

In this paragraph we recall some definitions and properties of the DGLA of
graded derivations from \cite{Frolicher56},\ \cite{Kodaira61} (see also \cite%
{Michor2008}).

\begin{notation}
Let $M$ be a smooth manifold. We denote by $\Lambda ^{\ast }M$ the algebra
of differential forms on $M$, by $\mathfrak{X}\left( M\right) $ the Lie
algebra of vector fields on $M$ and by $\Lambda ^{\ast }M\otimes TM$ the
algebra of $TM$-valued differential form on $M$, where $TM$ is the tangent
bundle to $M$. In the sequel, we will identify $\Lambda ^{1}M\otimes TM$
with the algebra $End\left( TM\right) $ of endomorphisms of $TM$ by their
canonical isomorphism: for $\sigma \in \Lambda ^{1}M$, $X,Y\in \mathfrak{X}%
\left( M\right) $, $\left( \sigma \otimes X\right) \left( Y\right) =\sigma
\left( Y\right) X$.
\end{notation}

\begin{definition}
A differential graded Lie agebra (DGLA) is a triple $\left( V^{\ast },d,%
\left[ \cdot ,\cdot \right] \right) $ such that:

1) $V^{\ast }=\oplus _{i\in \mathbb{N}}V^{i}$, where $\left( V^{i}\right)
_{i\in \mathbb{N}}$ \ is a family of $\mathbb{C}$-vector spaces and $%
d:V^{\ast }\rightarrow V^{\ast }$is a graded homomorphism such that $d^{2}=0$%
. An element $a\in V^{k}$ is said to be homogeneous of degree $k=\deg a$.

2) $\left[ \cdot ,\cdot \right] :$ $V^{\ast }\times V^{\ast }\rightarrow
V^{\ast }$defines a structure of graded Lie algebra i.e. for homogeneous
elements we have

\begin{equation*}
\left[ a,b\right] =-\left( -1\right) ^{\deg a\deg b}\left[ b,a\right]
\end{equation*}%
and 
\begin{equation*}
\left[ a,\left[ b,c\right] \right] =\left[ \left[ a,b\right] ,c\right]
+\left( -1\right) ^{\deg a\deg b}\left[ b,\left[ a,c\right] \right]
\end{equation*}

3) $d$ is compatible with the graded Lie algebra structure i.e. 
\begin{equation*}
d\left[ a,b\right] =\left[ da,b\right] +\left( -1\right) ^{\deg a}\left[ a,db%
\right] .
\end{equation*}
\end{definition}

\begin{definition}
Let $\left( V^{\ast },d,\left[ \cdot ,\cdot \right] \right) $ be a DGLA and $%
a\in V^{1}$. We say that $a$ verifies the Maurer-Cartan equation in $\left(
V^{\ast },d,\left[ \cdot ,\cdot \right] \right) $ if%
\begin{equation*}
da+\frac{1}{2}\left[ a,a\right] =0.
\end{equation*}
\end{definition}

\begin{definition}
Let $A=\oplus _{k\in \mathbb{Z}}A_{k}$ be a graded algebra. A linear mapping 
$D:A\rightarrow A$ is called a graded derivation of degree $p=\left\vert
D\right\vert $ if $D:A_{k}\rightarrow A_{k+p}$ and $D\left( ab\right)
=D\left( a\right) b+\left( -1\right) ^{p\deg a}aD\left( b\right) $.
\end{definition}

\begin{definition}
Let $M$ be a smooth manifold. We denote by $\mathcal{D}^{\ast }\left(
M\right) $ the graded algebra of graded derivations of $\Lambda ^{\ast }M$.
\end{definition}

\begin{definition}
Let $P,Q$ be homogeneous elements of degree $\left\vert P\right\vert ,\
\left\vert Q\right\vert $ of $\mathcal{D}^{\ast }\left( M\right) $. We define%
\begin{equation*}
\left[ P,Q\right] =PQ-\left( -1\right) ^{\left\vert P\right\vert \left\vert
Q\right\vert }QP,
\end{equation*}%
\begin{equation*}
\daleth P=\left[ d,P\right] .
\end{equation*}
\end{definition}

\begin{lemma}
Let $M$ be a smooth manifold. Then $\left( \mathcal{D}^{\ast }\left(
M\right) ,\left[ \cdot ,\cdot \right] ,\daleth \right) $ is a DGLA.
\end{lemma}

\begin{definition}
Let $\alpha \in \Lambda ^{\ast }M$ and $X\in \mathfrak{X}\left( M\right) $.
We define $\mathcal{L}_{\alpha \otimes X}$,\ $\mathcal{I}_{\alpha \otimes X}$
by 
\begin{equation}
\mathcal{L}_{\alpha \otimes X}\sigma =\alpha \wedge \mathcal{L}_{X}\sigma
+\left( -1\right) ^{\left\vert \alpha \right\vert }d\alpha \wedge \iota
_{X}\sigma ,\ \sigma \in \Lambda ^{\ast }\left( M\right)
\label{L alfa tens X}
\end{equation}%
\begin{equation}
\mathcal{I}_{\alpha \otimes X}\sigma =\alpha \wedge \iota _{X}\sigma ,\
\sigma \in \Lambda ^{\ast }\left( M\right)  \label{I alfa tens X}
\end{equation}%
where $\mathcal{L}_{X}$ is the Lie derivative and $\iota _{X}$ the
contraction by $X$.

For $\Phi \in \Lambda ^{\ast }M\otimes TM$ we define $\mathcal{L}_{\Phi },\ 
\mathcal{I}_{\Phi }$ as the extensions by linearity of (\ref{L alfa tens X}%
), (\ref{I alfa tens X}).
\end{definition}

\begin{remark}
\label{Ifisigma(X,Y)}Let $\omega \in \Lambda ^{2}\left( M\right) $, $Z\in 
\mathfrak{X}\left( M\right) $ and $\sigma \in \Lambda ^{1}\left( M\right) $.
Then for every $X,Y\in \mathfrak{X}\left( M\right) $%
\begin{equation*}
\mathcal{I}_{\omega \otimes Z}\sigma \left( X,Y\right) =\left( \omega \wedge
\iota _{Z}\sigma \right) \left( X,Y\right) =\sigma \left( Z\right) \omega
\left( X,Y\right) =\sigma \left( \left( \omega \otimes Z\right) \left(
X,Y\right) \right) .
\end{equation*}%
By linearity, for every $\Phi \in \Lambda ^{2}M\otimes TM$, $\sigma \in
\Lambda ^{1}\left( M\right) $, $X,Y\in \mathfrak{X}\left( M\right) $ we have%
\begin{equation*}
\mathcal{I}_{\Phi }\sigma \left( X,Y\right) =\sigma \left( \Phi \left(
X,Y\right) \right) .
\end{equation*}
\end{remark}

\begin{lemma}
For every $\Phi \in \Lambda ^{k}M\otimes TM$, $\mathcal{L}_{\Phi }$,\ $%
\mathcal{I}_{\Phi }\in \mathcal{D}^{\ast }\left( M\right) $, $\left\vert 
\mathcal{L}_{\Phi }\right\vert =k$,\ $\left\vert \mathcal{I}_{\Phi
}\right\vert =k-1$.
\end{lemma}

\begin{notation}
\begin{equation*}
\mathcal{L}\left( M\right) =\left\{ \mathcal{L}_{\Phi }:\ \Phi \in \Lambda
^{\ast }M\otimes TM\right\} ,\ \mathcal{I}\left( M\right) =\left\{ \mathcal{I%
}_{\Phi }:\ \Phi \in \Lambda ^{\ast }M\otimes TM\right\} .
\end{equation*}%
In \cite{Frolicher56} the graded derivations of $\mathcal{L}\left( M\right) $
(respectively of $\mathcal{I}\left( M\right) $) are called of type $d_{\ast
} $ (respectively of type $\iota _{\ast }$).
\end{notation}

\begin{lemma}
\begin{enumerate}
\item \label{Lemme technique} For every $D\in \mathcal{D}^{k}\left( M\right) 
$ there exist unique forms $\Phi \in \Lambda ^{k}M\otimes TM,\Psi \in
\Lambda ^{k+1}M\otimes TM$ such that 
\begin{equation*}
D=\mathcal{L}_{\Phi }+\mathcal{I}_{\Psi },
\end{equation*}%
so%
\begin{equation*}
\mathcal{D}^{\ast }\left( M\right) =\mathcal{L}\left( M\right) \oplus 
\mathcal{I}\left( M\right) .
\end{equation*}%
We denote $\mathcal{L}_{\Phi }=\mathfrak{L}\left( D\right) $ and $\mathcal{I}%
_{\Psi }=\mathfrak{I}\left( D\right) $

\item \label{daled ferme}For every $\Phi \in \Lambda ^{\ast }M\otimes TM$%
\begin{equation}
\daleth \left( -1\right) ^{\left\vert \Phi \right\vert }\mathcal{I}_{\Phi }=%
\left[ \mathcal{I}_{\Phi },d\right] =\mathcal{L}_{\Phi }.
\label{[Ifi,d]=Lfi}
\end{equation}

\item 
\begin{equation*}
\mathcal{L}\left( M\right) =\ker \daleth .
\end{equation*}
\end{enumerate}
\end{lemma}

\begin{notation}
We denote by $\aleph :\mathcal{D}^{\ast }\left( M\right) \rightarrow 
\mathcal{D}^{\ast }\left( M\right) $ the mapping defined by 
\begin{equation*}
\aleph \left( D\right) =\left( -1\right) ^{\left\vert D\right\vert }\mathcal{%
I}\left( D\right)
\end{equation*}
\end{notation}

\begin{remark}
\begin{equation*}
Id=\daleth \aleph +\aleph \daleth .
\end{equation*}%
Indeed for $D=\mathcal{L}_{\Phi }+\mathcal{I}_{\Psi }\in \mathcal{D}^{\ast
}\left( M\right) ,$ by using Lemma \ref{daled ferme} we have%
\begin{equation*}
\left( \daleth \aleph +\aleph \daleth \right) \left( D\right) =\left(
\daleth \aleph +\aleph \daleth \right) \left( \mathcal{L}_{\Phi }+\mathcal{I}%
_{\Psi }\right) =\daleth \left( -1\right) ^{\left\vert \Phi \right\vert }%
\mathcal{I}_{\Phi }+\aleph \left( -1\right) ^{\left\vert \Phi \right\vert }%
\mathcal{L}_{\Psi }=\mathcal{L}_{\Phi }+\mathcal{I}_{\Psi }=D.
\end{equation*}
\end{remark}

\begin{lemma}
\begin{enumerate}
\item \label{Tensorial}Let $D\in \mathcal{D}^{\ast }\left( M\right) $. The
following are equivalent:

i) $D\in \mathcal{I}\left( M\right) $;

ii) $D\left\vert \Lambda ^{0}\left( M\right) \right. =0$;

iii) $D\left( f\omega \right) =fD\left( \omega \right) $ for every $f\in
C^{\infty }\left( M\right) $ and $\omega \in \Lambda ^{\ast }\left( M\right) 
$.

\item The mapping $\mathcal{L}:\Lambda ^{\ast }M\otimes TM\rightarrow 
\mathcal{D}^{\ast }\left( M\right) $ defined by $\mathcal{L}\left( \Phi
\right) =\mathcal{L}_{\Phi }$ is an injective morphism of graded Lie
algebras.
\end{enumerate}
\end{lemma}

\begin{remark}
\label{d=lId}$d\in \mathcal{D}^{1}\left( M\right) $ and%
\begin{equation*}
d=\mathcal{L}_{Id_{T\left( M\right) }}=-\daleth \mathcal{I}_{Id_{T\left(
M\right) }}.
\end{equation*}
\end{remark}

By Lemma \ref{Lemme technique} and the Jacobi identity, for every $\Phi \in
\Lambda ^{k}M\otimes TM,\Psi \in \Lambda ^{l}M\otimes TM$ we have 
\begin{eqnarray*}
\daleth \left( \left[ \mathcal{L}_{\Phi },\mathcal{L}_{\Psi }\right] \right)
&=&\left[ d,\left[ \mathcal{L}_{\Phi },\mathcal{L}_{\Psi }\right] \right] =%
\left[ \left[ d,\mathcal{L}_{\Phi }\right] ,\mathcal{L}_{\Psi }\right]
+\left( -1\right) ^{\left\vert \Phi \right\vert }\left[ \mathcal{L}_{\Phi },%
\left[ d,\mathcal{L}_{\Psi }\right] \right] \\
&=&\left[ \daleth \mathcal{L}_{\Phi },\mathcal{L}_{\Psi }\right] +\left(
-1\right) ^{\left\vert \Phi \right\vert }\left[ \mathcal{L}_{\Phi },\daleth 
\mathcal{L}_{\Phi }\right] =0,
\end{eqnarray*}%
so there exists a unique form $\left[ \Phi ,\Psi \right] \in \Lambda
^{k+l}M\otimes TM$ such that 
\begin{equation}
\left[ \mathcal{L}_{\Phi },\mathcal{L}_{\Psi }\right] =\mathcal{L}_{\left[
\Phi ,\Psi \right] }.  \label{Fr N bracket}
\end{equation}%
This gives the following

\begin{definition}
Let $\Phi ,\Psi \in \Lambda ^{\ast }M\otimes TM$. The Fr\"{o}%
licher-Nijenhuis bracket of $\Phi $ and $\Psi $ is the unique form $\left[
\Phi ,\Psi \right] _{\mathcal{FN}}\in \Lambda ^{\ast }M\otimes TM$ verifying
(\ref{Fr N bracket}).
\end{definition}

\begin{lemma}
\label{[Lfi,ICsi]}Let $\Phi _{1}\in \Lambda ^{k_{1}}M\otimes TM,\Phi _{2}\in
\Lambda ^{k_{2}}M\otimes TM,\Psi _{1}\in \Lambda ^{k_{1}+1}M\otimes TM,\Psi
_{2}\in \Lambda ^{k_{2}+1}M\otimes TM$. Then%
\begin{eqnarray*}
\left[ \mathcal{L}_{\Phi _{1}}+\mathcal{I}_{\Psi _{1}},\mathcal{L}_{\Phi
_{2}}+\mathcal{I}_{\Psi _{2}}\right] &=&\mathcal{L}_{\left[ \Phi _{1},\Phi
_{2}\right] _{\mathcal{FN}}+\mathcal{I}_{\Psi _{1}}\Phi _{2}-\left(
-1\right) ^{k_{1}k_{2}}\mathcal{I}_{\Psi _{2}}\Phi _{1}} \\
&&+\mathcal{I}_{\mathcal{I}_{\Psi _{1}}\Psi _{2}-\left( -1\right)
^{k_{1}k_{2}}\mathcal{I}_{\Psi _{2}}\Psi _{1}+\left[ \Phi _{1},\Psi _{2}%
\right] _{\mathcal{FN}}-\left( -1\right) ^{k_{1}k_{2}}\left[ \Phi _{2},\Psi
_{1}\right] _{\mathcal{FN}}}.
\end{eqnarray*}%
In particular 
\begin{equation}
\left[ \mathcal{I}_{\Phi },\mathcal{I}_{\Psi }\right] =\mathcal{I}_{\mathcal{%
I}_{\Phi }\Psi }-\left( -1\right) ^{\left( \left\vert \Phi \right\vert
+1\right) \left( \left\vert \Psi \right\vert +1\right) }\mathcal{I}_{%
\mathcal{I}_{\Psi }\Phi };  \label{[I,I]}
\end{equation}%
\begin{equation}
\left[ \mathcal{L}_{\Phi },\mathcal{I}_{\Psi }\right] =\mathcal{I}_{\left[
\Phi ,\Psi \right] _{\mathcal{FN}}}-\left( -1\right) ^{\left\vert \Phi
\right\vert \left( \left\vert \Psi \right\vert +1\right) }\mathcal{L}_{%
\mathcal{I}_{\Psi }\Phi };  \label{[L,I]}
\end{equation}%
\begin{equation*}
\left[ \mathcal{I}_{\Psi },\mathcal{L}_{\Phi }\right] =\mathcal{L}_{\mathcal{%
I}_{\Psi }\Phi }-\left( -1\right) ^{\left\vert \Phi \right\vert }\mathcal{I}%
_{\left[ \Psi ,\Phi \right] _{\mathcal{FN}}}.
\end{equation*}
\end{lemma}

\begin{definition}
Let $\Phi \in \Lambda ^{1}M\otimes TM$. The Nijenhuis tensor of $\Phi $ is $%
N_{\Phi }\in \Lambda ^{2}M\otimes TM$ defined by%
\begin{equation*}
N_{\Phi }\left( X,Y\right) =\left[ \Phi X,\Phi Y\right] +\Phi ^{2}\left[ X,Y%
\right] -\Phi \left[ \Phi X,Y\right] -\Phi \left[ X,\Phi Y\right] ,\ X,Y\in 
\mathfrak{X}\left( M\right) .
\end{equation*}
\end{definition}

\begin{proposition}
\label{Calcul [,]FN}Let $\alpha \in \Lambda ^{k}\left( M\right) $, $\beta
\in \Lambda ^{l}\left( M\right) $, $X,Y\in \mathfrak{X}\left( M\right) $%
.Then:

\begin{enumerate}
\item 
\begin{eqnarray*}
\left[ \alpha \otimes X,\beta \otimes Y\right] _{\mathcal{FN}} &=&\alpha
\wedge \beta \otimes \left[ X,Y\right] +\alpha \wedge \mathcal{L}_{X}\beta
\otimes Y-\mathcal{L}_{Y}\alpha \wedge \beta \otimes X \\
&&+\left( -1\right) ^{k}\left( d\alpha \wedge \iota _{X}\beta \otimes
Y+\iota _{Y}\alpha \wedge d\beta \otimes X\right) .
\end{eqnarray*}

\item Let $\Phi ,\Psi \in \Lambda ^{1}M\otimes TM$. Then%
\begin{eqnarray*}
\left[ \Phi ,\Psi \right] _{\mathcal{FN}}\left( X,Y\right) &=&\left[ \Phi
X,\Psi Y\right] -\left[ \Phi Y,\Psi X\right] -\Psi \left[ \Phi X,Y\right]
+\Psi \left[ \Phi Y,X\right] \\
&&-\Phi \left[ \Psi X,Y\right] +\Phi \left[ \Psi Y,X\right] +\frac{1}{2}\Psi
\left( \Phi \left[ X,Y\right] \right) -\frac{1}{2}\Psi \left( \Phi \left[ Y,X%
\right] \right) \\
&&+\frac{1}{2}\Phi \left( \Psi \left[ X,Y\right] \right) -\frac{1}{2}\Phi
\left( \Psi \left[ Y,X\right] \right) \\
&=&\left[ \Phi X,\Psi Y\right] +\left[ \Psi X,\Phi Y\right] +\Phi \left(
\Psi \left[ X,Y\right] \right) +\Psi \left( \Phi \left[ X,Y\right] \right) \\
&&-\Phi \left[ \Psi X,Y\right] -\Phi \left[ X,\Psi Y\right] -\Psi \left[
\Phi X,Y\right] -\Psi \left[ X,\Phi Y\right] .
\end{eqnarray*}%
In particular%
\begin{equation*}
\left[ \Phi ,\Phi \right] _{\mathcal{FN}}\left( X,Y\right) =2\left( \left[
\Phi X,\Phi Y\right] +\Phi ^{2}\left[ X,Y\right] -\Phi \left[ \Phi X,Y\right]
-\Phi \left[ X,\Phi Y\right] \right) =2N_{\Phi }\left( X,Y\right) .
\end{equation*}
\end{enumerate}
\end{proposition}

\section{Canonical solutions of Maurer-Cartan equation}

\begin{definition}
\label{action} Let $\Phi \in \Lambda ^{1}M\otimes TM$.

a) Let $\sigma \in \Lambda ^{p}M$. We define $\Phi \sigma \in \Lambda ^{p}M$
by $\Phi \sigma =\sigma $ if $p=0$ and%
\begin{equation*}
\left( \Phi \sigma \right) \left( V_{1},\cdot \cdot \cdot ,V_{p}\right)
=\sigma \left( \Phi V_{1},\cdot \cdot \cdot ,\Phi V_{p}\right) \ if\
p\geqslant 1,\ V_{1},\cdot \cdot \cdot ,V_{p}\in \mathfrak{X}\left( M\right)
.
\end{equation*}%
b) Let $\Psi \in \Lambda ^{p}M\otimes TM$. We define $\Phi \Psi \in \Lambda
^{p}M\otimes TM$ by $\Phi \Psi =\Psi $ if $p=0$ and 
\begin{equation*}
\Phi \Psi \left( V_{1},...,V_{p}\right) =\Phi \left( \Psi \left(
V_{1},...,V_{p}\right) \right) ,\ V_{1},...,V_{p}\in \mathfrak{X}\left(
M\right) \ if\ p\geqslant 1.
\end{equation*}
\end{definition}

\begin{lemma}
\label{I=produit}Let $\Phi \in \Lambda ^{1}M\otimes TM$, $\Psi \in \Lambda
^{2}M\otimes TM$. Then 
\begin{equation*}
\mathcal{I}_{\Psi }\Phi =\Phi \Psi .
\end{equation*}
\end{lemma}

\begin{proof}
It is sufficient to prove the assertion for $\Psi =\alpha \otimes X$, $\Phi
=\beta \otimes Y$, $\alpha \in \Lambda ^{2}\left( M\right) $, $\beta \in
\Lambda ^{1}\left( M\right) $, $X,Y\in \mathfrak{X}\left( M\right) $.

For every $Z_{1},Z_{2}\in \mathfrak{X}\left( M\right) $ we have 
\begin{equation*}
\left( \beta \otimes Y\right) \left( \alpha \otimes X\right) \left(
Z_{1},Z_{2}\right) =\left( \beta \otimes Y\right) \left( \alpha \left(
Z_{1},Z_{2}\right) \otimes X\right) =\beta \left( X\right) \alpha \left(
Z_{1},Z_{2}\right) \otimes Y.
\end{equation*}%
Since 
\begin{equation*}
\mathcal{I}_{\alpha \otimes X}\beta \otimes Y=\left( \mathcal{I}_{\alpha
\otimes X}\beta \right) \otimes Y=\beta \left( X\right) \alpha \otimes Y,
\end{equation*}%
the Lemma is proved.
\end{proof}

\begin{definition}
Let $D\in \mathcal{D}^{k}\left( M\right) $ and $\Phi \in \Lambda
^{1}M\otimes TM$ invertible. We define $\Phi ^{-1}D\Phi :\Lambda ^{\ast
}M\rightarrow \Lambda ^{\ast }M$ by $\Phi ^{-1}D\Phi \left( \sigma \right)
=\Phi ^{-1}D\left( \Phi \sigma \right) $.
\end{definition}

\begin{lemma}
\label{Fi-1DFi}Let $D\in \mathcal{D}^{k}\left( M\right) $ and $\Phi \in
\Lambda ^{1}M\otimes TM$ invertible. Then $\Phi ^{-1}D\Phi \in \mathcal{D}%
^{k}\left( M\right) .$
\end{lemma}

\begin{proof}
Let $\sigma \in \Lambda ^{p}\left( M\right) $, $\eta \in \Lambda ^{q}\left(
M\right) $. Since $\Phi \left( \sigma \wedge \eta \right) =\Phi \sigma
\wedge \Phi \eta $, it follows that%
\begin{eqnarray*}
\left( \Phi ^{-1}D\Phi \right) \left( \sigma \wedge \eta \right) &=&\Phi
^{-1}D\left( \Phi \left( \sigma \wedge \eta \right) \right) =\Phi
^{-1}D\left( \Phi \sigma \wedge \Phi \eta \right) \\
&=&\Phi ^{-1}\left( D\Phi \sigma \wedge \Phi \eta +\left( -1\right)
^{pk}\Phi \sigma \wedge D\Phi \eta \right) \\
&=&\Phi ^{-1}D\Phi \sigma \wedge \eta +\left( -1\right) ^{pk}\sigma \wedge
\Phi ^{-1}D\Phi \eta .
\end{eqnarray*}
\end{proof}

\begin{notation}
Let $\Phi \in \Lambda ^{1}M\otimes TM$ such that $R_{\Phi }=Id_{TM}+\Phi $
is invertible. Set 
\begin{equation*}
d_{\Phi }=R_{\Phi }dR_{\Phi }^{-1},
\end{equation*}%
\begin{equation*}
e_{\Phi }=d_{\Phi }-d
\end{equation*}%
and%
\begin{equation*}
b\left( \Phi \right) =-\frac{1}{2}R_{\Phi }^{-1}\left[ \Phi ,\Phi \right] _{%
\mathcal{FN}}.
\end{equation*}
\end{notation}

The following Theorem is a refinement of results from \cite%
{BartolomeisMatveev13} and \cite{BartolomeisTomassini13} :

\begin{theorem}
\label{Calcul RFIdRFi-1}Let $\Phi \in \Lambda ^{1}M\otimes TM$ such that $%
R_{\Phi }=Id_{T\left( M\right) }+\Phi $ is invertible. Then 
\begin{equation}
e_{\Phi }=\mathcal{L}_{\Phi }+\mathcal{I}_{b\left( \Phi \right) }.
\label{efi=Lfi+Ibfi}
\end{equation}
\end{theorem}

\begin{proof}
Since both terms of (\ref{efi=Lfi+Ibfi}) are derivations of degree $1$, it
is enough to prove (\ref{efi=Lfi+Ibfi}) on $\Lambda ^{0}\left( M\right) $
and $\Lambda ^{1}\left( M\right) $.

Let $f\in \Lambda ^{0}\left( M\right) $ and $X\in \mathfrak{X}\left(
M\right) $. Then%
\begin{equation}
d_{\Phi }f\left( X\right) =\left( R_{\Phi }dR_{\Phi }^{-1}f\right) \left(
X\right) =\left( R_{\Phi }df\right) \left( X\right) =df\left( Id_{TM}+\Phi
\right) \left( X\right) =df\left( X\right) +df\left( \Phi \left( X\right)
\right) .  \label{dfif(X)}
\end{equation}%
If $\alpha \in \Lambda ^{1}\left( M\right) $,\ $X\in \mathfrak{X}\left(
M\right) $, by (\ref{I alfa tens X}), 
\begin{equation*}
\mathcal{I}_{\alpha \otimes Y}\left( df\right) \left( X\right) =\left(
\alpha \otimes \iota _{Y}df\right) \left( X\right) =df\left( Y\right) \left(
\alpha \left( X\right) \right) =df\left( \alpha \otimes Y\right) X
\end{equation*}%
and by linearity we obtain%
\begin{equation*}
\mathcal{I}_{\Phi }\left( df\right) \left( X\right) =df\left( \Phi \left(
X\right) \right) .
\end{equation*}%
So, from (\ref{dfif(X)}) it follows that%
\begin{equation*}
d_{\Phi }f\left( X\right) =df\left( X\right) +\mathcal{I}_{\Phi }\left(
df\right) \left( X\right) =\left( d+\left[ \mathcal{I}_{\Phi },d\right]
\right) f\left( X\right) =\left( d+\mathcal{L}_{\Phi }\right) f\left(
X\right) =\left( d+\mathcal{L}_{\Phi }\right) f\left( X\right) .
\end{equation*}%
Since $\mathcal{I}_{\Phi }$ is of type $i_{\ast }$, $\mathcal{I}_{\Phi }f=0$
and therefore (\ref{efi=Lfi+Ibfi}) is verified for every $f\in \Lambda
^{0}\left( M\right) $.

Let now $\sigma \in \Lambda ^{1}\left( M\right) $.

We will prove firstly that%
\begin{equation}
\mathcal{I}_{b\left( \Phi \right) }\left( \sigma \right) \left( X,Y\right)
=-\sigma \left( \frac{1}{2}R_{\Phi }^{-1}N_{R_{\Phi }}\left( X,Y\right)
\right) .  \label{Ibfisigma(X,Y)}
\end{equation}%
By using Remark \ref{d=lId}, we have 
\begin{equation*}
\mathcal{L}_{\left[ Id_{T\left( M\right) },Id_{T\left( M\right) }\right] }=%
\left[ \mathcal{L}_{Id_{T\left( M\right) }},\mathcal{L}_{Id_{T\left(
M\right) }}\right] =\left[ d,d\right] =0
\end{equation*}%
and%
\begin{equation*}
\left[ \mathcal{L}_{\Phi },\mathcal{L}_{Id_{T\left( M\right) }}\right] =%
\left[ \mathcal{L}_{\Phi },d\right] =\daleth \mathcal{L}_{\Phi }=0.
\end{equation*}%
So%
\begin{equation*}
\left[ R_{\Phi },R_{\Phi }\right] _{\mathcal{FN}}=\left[ Id_{T\left(
M\right) }+\Phi ,Id_{T\left( M\right) }+\Phi \right] _{\mathcal{FN}}=\left[
\Phi ,\Phi \right] _{\mathcal{FN}}
\end{equation*}%
and by Proposition \ref{Calcul [,]FN}%
\begin{equation}
2N_{R_{\Phi }}=\left[ R_{\Phi },R_{\Phi }\right] _{\mathcal{FN}}=\left[ \Phi
,\Phi \right] _{\mathcal{FN}}=2N_{\Phi }.  \label{[Rfi,Rfi]=Nijenhuys}
\end{equation}

By Remark \ref{Ifisigma(X,Y)} it follows that 
\begin{equation*}
\mathcal{I}_{-\frac{1}{2}R_{\Phi }^{-1}\left[ \Phi ,\Phi \right] _{\mathcal{%
FN}}}\left( \sigma \right) \left( X,Y\right) =-\sigma \left( \frac{1}{2}%
R_{\Phi }^{-1}\left[ \Phi ,\Phi \right] _{\mathcal{FN}}\left( X,Y\right)
\right) =-\sigma \left( \frac{1}{2}R_{\Phi }^{-1}N_{R_{\Phi }}\left(
X,Y\right) \right)
\end{equation*}%
and (\ref{Ibfisigma(X,Y)}) is proved.

We will compute now $\mathcal{L}_{\Phi }\sigma $, $d_{\Phi }\sigma $ and $%
\mathcal{I}_{b\left( \Phi \right) }\sigma $:

We remark that (\ref{[Ifi,d]=Lfi}) gives $\mathcal{L}_{\Phi }\sigma =\left[ 
\mathcal{I}_{\Phi },d\right] \left( \sigma \right) $ and thus

\begin{eqnarray}
\mathcal{L}_{\Phi }\sigma &=&\left[ \mathcal{I}_{\Phi },d\right] \left(
\sigma \right) \left( X,Y\right) =\left( \mathcal{I}_{\Phi }d\sigma \right)
\left( X,Y\right) -d\left( \mathcal{I}_{\Phi }\sigma \right) \left(
X,Y\right)  \notag \\
&=&d\sigma \left( \Phi X,Y\right) +d\sigma \left( X,\Phi Y\right)  \notag \\
&&-X\left( \left( \mathcal{I}_{\Phi }\sigma \right) \left( Y\right) \right)
+Y\left( \left( \mathcal{I}_{\Phi }\sigma \right) \left( X\right) \right)
+\left( \mathcal{I}_{\Phi }\sigma \right) \left[ X,Y\right]  \notag \\
&=&\left( \Phi X\right) \left( \sigma \left( Y\right) \right) -Y\left(
\sigma \left( \Phi X\right) \right) -\sigma \left( \left[ \Phi X,Y\right]
\right)  \notag \\
&&+X\left( \sigma \left( \Phi Y\right) \right) -\left( \Phi Y\right) \left(
\sigma \left( X\right) \right) -\sigma \left( \left[ X,\Phi Y\right] \right)
\notag \\
&&-X\left( \sigma \left( \Phi \left( Y\right) \right) \right) +Y\left(
\sigma \left( \Phi \left( X\right) \right) \right) +\sigma \left( \Phi
\left( \left[ X,Y\right] \right) \right)  \notag \\
&=&\left( \Phi X\right) \left( \sigma \left( Y\right) \right) -\left( \Phi
Y\right) \left( \sigma \left( X\right) \right) +\sigma \left( \Phi \left( 
\left[ X,Y\right] \right) -\left[ \Phi X,Y\right] -\left[ X,\Phi Y\right]
\right) ,  \label{Calcul d_R 1}
\end{eqnarray}

\begin{eqnarray}
d_{\Phi }\sigma \left( X,Y\right) &=&\left( R_{\Phi }dR_{\Phi }^{-1}\right)
\sigma \left( X,Y\right) =\left( dR_{\Phi }^{-1}\sigma \right) \left(
R_{\Phi }X,R_{\Phi }Y\right)  \notag \\
&=&\left( R_{\Phi }X\right) \left( \left( R_{\Phi }^{-1}\sigma \right)
\left( R_{\Phi }Y\right) \right) -\left( R_{\Phi }Y\right) \left( \left(
R_{\Phi }^{-1}\sigma \right) \left( R_{\Phi }X\right) \right) -R_{\Phi
}^{-1}\sigma \left( \left[ R_{\Phi }X,R_{\Phi }Y\right] \right)  \notag \\
&=&\left( R_{\Phi }X\right) \left( \sigma Y\right) -\left( R_{\Phi }Y\right)
\left( \sigma X\right) -\sigma \left( R_{\Phi }^{-1}\left( \left[ R_{\Phi
}X,R_{\Phi }Y\right] \right) \right)  \notag \\
&=&\left( X+\Phi X\right) \left( \sigma Y\right) -\left( Y+\Phi Y\right)
\left( \sigma X\right) -\sigma \left( R_{\Phi }^{-1}\left( \left[ R_{\Phi
}X,R_{\Phi }Y\right] \right) \right)  \notag \\
&=&\left( \Phi X\right) \left( \sigma Y\right) -\left( \Phi Y\right) \left(
\sigma X\right) +X\left( \sigma Y\right) -Y\left( \sigma X\right) -\sigma
\left( R_{\Phi }^{-1}\left( \left[ R_{\Phi }X,R_{\Phi }Y\right] \right)
\right) .  \notag \\
&&  \label{Calcul d_R 2}
\end{eqnarray}

By developping (\ref{Ibfisigma(X,Y)}) we have%
\begin{eqnarray}
\left( \mathcal{I}_{b\left( \Phi \right) }\sigma \right) \left( X,Y\right)
&=&-\sigma \left( R_{\Phi }^{-1}N_{R_{\Phi }}\left( X,Y\right) \right) 
\notag \\
&=&\sigma \left( R_{\Phi }^{-1}\left( \left[ R_{\Phi }X,R_{\Phi }Y\right]
\right) +R_{\Phi }\left[ X,Y\right] -\left[ R_{\Phi }X,Y\right] -\left[
X,R_{\Phi }Y\right] \right)  \notag \\
&=&\sigma \left( R_{\Phi }^{-1}\left( \left[ R_{\Phi }X,R_{\Phi }Y\right]
\right) \right) +\sigma \left( \left[ X,Y\right] +\Phi \left[ X,Y\right]
\right)  \notag \\
&&-\sigma \left( \left[ X,Y\right] +\left[ \Phi X,Y\right] +\left[ X,Y\right]
+\left[ X,\Phi Y\right] \right)  \notag \\
&=&\sigma \left( R_{\Phi }^{-1}\left( \left[ R_{\Phi }X,R_{\Phi }Y\right]
\right) \right) +\sigma \left( \Phi \left[ X,Y\right] -\left[ X,Y\right] -%
\left[ \Phi X,Y\right] -\left[ X,\Phi Y\right] \right)  \label{Calcul d_R 3}
\\
&&-\sigma \left( \left[ X,Y\right] +\left[ \Phi X,Y\right] +\left[ X,\Phi Y%
\right] \right) .  \notag
\end{eqnarray}%
Since%
\begin{equation*}
d\sigma \left( X,Y\right) =X\left( \sigma Y\right) -Y\left( \sigma X\right)
-\sigma \left[ X,Y\right] ,
\end{equation*}%
by comparing (\ref{Calcul d_R 1}),\ (\ref{Calcul d_R 2}) and (\ref{Calcul
d_R 3}) it follows that (\ref{efi=Lfi+Ibfi}) is verified for each form in $%
\Lambda ^{1}M$ and the Lemma is proved.
\end{proof}

\begin{theorem}
\label{Solution MC}Let $\Phi \in \Lambda ^{1}M\otimes TM$ such that $R_{\Phi
}=Id_{T\left( M\right) }+\Phi $ is invertible. Then:

$i)$ $e_{\Phi }$ is a solution of the Maurer-Cartan equation in $\left( 
\mathcal{D}^{\ast }\left( M\right) ,\left[ \cdot ,\cdot \right] ,\daleth
\right) $.

$ii)$ Let $\Psi \in \Lambda ^{2}M\otimes TM$ such that $D=\mathcal{L}_{\Phi
}+\mathcal{I}_{\Psi }$ is a solution of the Maurer-Cartan equation in $%
\left( \mathcal{D}^{\ast }\left( M\right) ,\left[ \cdot ,\cdot \right]
,\daleth \right) $.) Than $\Psi =b\left( \Phi \right) $.
\end{theorem}

\begin{proof}
$i)$ Since $\left[ d,d\right] =0$, $\left[ R_{\Phi }dR_{\Phi }^{-1},R_{\Phi
}dR_{\Phi }^{-1}\right] =0$, and $\left[ d,R_{\Phi }dR_{\Phi }^{-1}\right] =%
\left[ R_{\Phi }dR_{\Phi }^{-1},d\right] $ it follows that%
\begin{eqnarray*}
\daleth e_{\Phi }+\frac{1}{2}\left[ e_{\Phi },e_{\Phi }\right] &=&\left[
d,R_{\Phi }dR_{\Phi }^{-1}-d\right] +\frac{1}{2}\left[ R_{\Phi }dR_{\Phi
}^{-1}-d,R_{\Phi }dR_{\Phi }^{-1}-d\right] \\
&=&\left[ d,R_{\Phi }dR_{\Phi }^{-1}\right] -\left[ d,R_{\Phi }dR_{\Phi
}^{-1}\right] =0.
\end{eqnarray*}

$ii)$ Let $D=\mathcal{L}_{\Phi }+\mathcal{I}_{\Psi }$, $\Phi \in \Lambda
^{1}M\otimes TM$,\ $\Psi \in \Lambda ^{2}M\otimes TM$. By using Lemma \ref%
{Lemme technique} and Lemma \ref{[Lfi,ICsi]} we have%
\begin{equation*}
\daleth D=\daleth \mathcal{I}_{\Psi }=\mathcal{L}_{\Psi }
\end{equation*}%
and 
\begin{eqnarray*}
\left[ D,D\right] &=&\left[ \mathcal{L}_{\Phi }+\mathcal{I}_{\Psi },\mathcal{%
L}_{\Phi }+\mathcal{I}_{\Psi }\right] =\mathcal{L}_{\left[ \Phi ,\Phi \right]
_{\mathcal{FN}}}+2\left[ \mathcal{L}_{\Phi },\mathcal{I}_{\Psi }\right] +%
\left[ \mathcal{I}_{\Psi },\mathcal{I}_{\Psi }\right] \\
&=&\mathcal{L}_{\left[ \Phi ,\Phi \right] _{\mathcal{FN}}}+2\left( \mathcal{I%
}_{\left[ \Phi ,\Psi \right] _{\mathcal{FN}}}+\mathcal{L}_{\mathcal{I}_{\Psi
}\Phi }\right) +\left[ \mathcal{I}_{\Psi },\mathcal{I}_{\Psi }\right] ,
\end{eqnarray*}%
so%
\begin{equation*}
\daleth D+\frac{1}{2}\left[ D,D\right] =\mathcal{L}_{\Psi }+\frac{1}{2}%
\mathcal{L}_{\left[ \Phi ,\Phi \right] _{\mathcal{FN}}}+\left( \mathcal{I}_{%
\left[ \Phi ,\Psi \right] _{\mathcal{FN}}}+\mathcal{L}_{\mathcal{I}_{\Psi
}\Phi }\right) +\frac{1}{2}\left[ \mathcal{I}_{\Psi },\mathcal{I}_{\Psi }%
\right] .
\end{equation*}%
It follows that%
\begin{equation*}
\mathfrak{L}\left( \daleth D+\frac{1}{2}\left[ D,D\right] \right) =\mathcal{L%
}_{\Psi }+\frac{1}{2}\mathcal{L}_{\left[ \Phi ,\Phi \right] _{\mathcal{FN}}}+%
\mathcal{L}_{\mathcal{I}_{\Psi }\Phi }
\end{equation*}%
and%
\begin{equation*}
\mathfrak{I}\left( \daleth D+\frac{1}{2}\left[ D,D\right] \right) =\mathcal{I%
}_{\left[ \Phi ,\Psi \right] _{\mathcal{FN}}}+\frac{1}{2}\left[ \mathcal{I}%
_{\Psi },\mathcal{I}_{\Psi }\right] .
\end{equation*}

Suppose that $D=\mathcal{L}_{\Phi }+\mathcal{I}_{\Psi }$ verifies the
Maurer-Cartan equation. Then%
\begin{equation*}
0=\mathcal{L}\left( \daleth D+\frac{1}{2}\left[ D,D\right] \right) =\mathcal{%
L}\left( \Psi +\frac{1}{2}\left[ \Phi ,\Phi \right] _{\mathcal{FN}}+\mathcal{%
I}_{\Psi }\Phi \right) .
\end{equation*}%
Since $\mathcal{L}$ is injective, this implies%
\begin{equation*}
\Psi +\frac{1}{2}\left[ \Phi ,\Phi \right] _{\mathcal{FN}}+\mathcal{I}_{\Psi
}\Phi =0.
\end{equation*}%
By Lemma \ref{I=produit} we obtain%
\begin{equation*}
\Psi +\frac{1}{2}\left[ \Phi ,\Phi \right] _{\mathcal{FN}}+\Phi \Psi =0,
\end{equation*}%
which is equivalent to%
\begin{equation*}
\Psi =-\frac{1}{2}\left( Id_{TM}+\Phi \right) ^{-1}\left[ \Phi ,\Phi \right]
_{\mathcal{FN}}=b\left( \Phi \right) .
\end{equation*}
\end{proof}

\begin{definition}
Let $\Phi \in \Lambda ^{1}M\otimes TM$ such that $Id_{T\left( M\right)
}+\Phi $ is invertible. $e_{\Phi }$ is called the canonical solution of
Maurer-Cartan equation associated to $\Phi $.
\end{definition}

\section{Canonical solutions of finite type of Maurer-Cartan equation}

\begin{theorem}
\label{gamma k}Let $\Phi \in \Lambda ^{1}M\otimes TM$ small enough such that 
$\sum_{h=0}^{\infty }\left( -1\right) ^{h}\Phi ^{h}=\left( Id_{T\left(
M\right) }+\Phi \right) ^{-1}\in \Lambda ^{1}M\otimes TM$ and $e_{\Phi }$
the canonical solution of Maurer-Cartan equation associated to $\Phi $. Then

a) $e_{\Phi }=\sum_{k=1}^{\infty }\gamma _{k}$, where $\gamma _{k}\in 
\mathcal{D}^{1}\left( M\right) $ are defined by induction as%
\begin{equation*}
\gamma _{1}=\mathcal{L}_{\Phi },\ \gamma _{k}=-\left( -1\right) ^{k}\frac{1}{%
2}\sum_{\left( p,q\right) \in \mathbb{N}^{\ast },\ p+q=k}\aleph \left( \left[
\gamma _{p},\gamma _{q}\right] \right) ,\ k\geqslant 2.
\end{equation*}%
b) 
\begin{equation*}
\gamma _{k}=\left( -1\right) ^{k+1}\frac{1}{2}\mathcal{I}_{\Phi ^{k-2}\left[
\Phi ,\Phi \right] _{\mathcal{FN}}},\ k\geqslant 2.
\end{equation*}%
c)%
\begin{equation*}
\mathcal{I}_{b\left( \Phi \right) }=\sum_{k=2}^{\infty }\gamma _{k}.
\end{equation*}
\end{theorem}

\begin{proof}
We remark that for $r\geqslant 2$, $\gamma _{r}\in \mathcal{I}\left(
M\right) $, so by \ref{[I,I]} it follows that $\left[ \gamma _{p},\gamma _{q}%
\right] \in \mathcal{I}\left( M\right) $ for $p,q\geqslant 2$. Since $\aleph
\left\vert _{\mathcal{I}\left( M\right) }\right. =0$ we have $\aleph \left( %
\left[ \gamma _{p},\gamma _{q}\right] \right) =0$ for $p,q\geqslant 2$ and so%
\begin{equation}
\gamma _{r}=-\left( -1\right) ^{r}\frac{1}{2}\aleph \left( \left[ \gamma
_{1},\gamma _{r-1}\right] +\left[ \gamma _{1},\gamma _{r-1}\right] \right)
=-\left( -1\right) ^{r}\aleph \left[ \gamma _{1},\gamma _{r-1}\right] ,\
r\geqslant 2.  \label{Gama_r=Alef}
\end{equation}

We will show by induction that for every $r\geqslant 2$%
\begin{equation}
\gamma _{r}=-\left( -1\right) ^{r}\frac{1}{2}\mathcal{I}_{\Phi ^{r-2}\left[
\Phi ,\Phi \right] }.  \label{gama r}
\end{equation}%
Suppose that for every $r\geqslant 3$%
\begin{equation*}
\gamma _{r-1}=-\left( -1\right) ^{k}\frac{1}{2}\mathcal{I}_{\Phi ^{r-3}\left[
\Phi ,\Phi \right] }.
\end{equation*}%
Then%
\begin{equation}
\aleph \left[ \gamma _{1},\gamma _{r-1}\right] =\aleph \left[ \mathcal{L}%
_{\Phi },-\left( -1\right) ^{r-1}\frac{1}{2}\mathcal{I}_{\Phi ^{r-3}\left[
\Phi ,\Phi \right] }\right] ,  \label{alfa[gama1,gamar-1}
\end{equation}%
and by \ref{[L,I]} we have%
\begin{equation}
\left[ \mathcal{L}_{\Phi },\mathcal{I}_{\Phi ^{r-3}\left[ \Phi ,\Phi \right]
}\right] =\left[ \mathcal{I}_{\left[ \Phi ,\Phi ^{r-3}\left[ \Phi ,\Phi %
\right] _{\mathcal{FN}}\right] _{\mathcal{FN}}}-\left( -1\right)
^{\left\vert \Phi \right\vert \left( \left\vert \Phi ^{r-3}\left[ \Phi ,\Phi %
\right] \right\vert +1\right) }\mathcal{L}_{\mathcal{I}_{\Phi ^{r-3}\left[
\Phi ,\Phi \right] _{\mathcal{FN}}}\Phi }\right] .  \label{[Lfi,Ifir-3]}
\end{equation}%
So, from (\ref{alfa[gama1,gamar-1}) and (\ref{[Lfi,Ifir-3]}) we obtain 
\begin{eqnarray}
\aleph \left[ \gamma _{1},\gamma _{r-1}\right] &=&-\frac{1}{2}\left(
-1\right) ^{r-1}\aleph \left( \mathcal{L}_{\mathcal{I}_{\Phi ^{r-3}\left[
\Phi ,\Phi \right] _{\mathcal{FN}}}\Phi }\right)  \label{Alef=I} \\
&=&-\frac{1}{2}\left( -1\right) ^{r-1}\left( -1\right) ^{\left\vert _{%
\mathcal{I}_{\Phi ^{r-3}\left[ \Phi ,\Phi \right] _{\mathcal{FN}}}\Phi
}\right\vert }\mathcal{I}_{\mathcal{I}_{\Phi ^{r-3}\left[ \Phi ,\Phi \right]
_{\mathcal{FN}}}\Phi }  \notag \\
&=&-\frac{1}{2}\left( -1\right) ^{r}\mathcal{I}_{\mathcal{I}_{\Phi ^{r-3}%
\left[ \Phi ,\Phi \right] _{\mathcal{FN}}}\Phi }.  \notag
\end{eqnarray}%
But by Lemma \ref{I=produit} 
\begin{equation*}
\mathcal{I}_{\Phi ^{r-3}\left[ \Phi ,\Phi \right] _{\mathcal{FN}}}\Phi =\Phi
^{r-2}\left[ \Phi ,\Phi \right] _{\mathcal{FN}}
\end{equation*}%
and (\ref{gama r}) is verified.

It follows that%
\begin{equation*}
\mathcal{I}_{b\left( \Phi \right) }=\mathcal{I}_{-\frac{1}{2}%
\sum_{k=0}^{\infty }\left( -1\right) ^{k}\Phi ^{k}\left[ \Phi ,\Phi \right]
_{\mathcal{FN}}}=\sum_{k=0}^{\infty }-\left( -1\right) ^{k}\frac{1}{2}%
\mathcal{I}_{\Phi ^{k}\left[ \Phi ,\Phi \right] _{\mathcal{FN}%
}}=\sum_{k=2}^{\infty }\gamma _{k}.
\end{equation*}

By Theorem \ref{efi=Lfi+Ibfi}%
\begin{equation*}
e_{\Phi }=\mathcal{L}_{\Phi }+\mathcal{I}_{b\left( \Phi \right)
}=\sum_{k=1}^{\infty }\gamma _{k}
\end{equation*}%
and the Proposition is proved.
\end{proof}

\begin{definition}
Let $\Phi \in \Lambda ^{1}M\otimes TM$ small enough such that $%
\sum_{h=0}^{\infty }\left( -1\right) ^{h}\Phi ^{h}=\left( Id_{T\left(
M\right) }+\Phi \right) ^{-1}\in \Lambda ^{1}M\otimes TM$ and $e_{\Phi }$
the canonical solution of Maurer-Cartan equation associated to $\Phi $. $%
e_{\Phi }$ is called of finite type if there exists $r\in \mathbb{N}$ if $%
\Phi ^{r}\left[ \Phi ,\Phi \right] _{\mathcal{FN}}=0$ and of finite type $r$
if $r=\min \left\{ s\in \mathbb{N}:\ \Phi ^{s}\left[ \Phi ,\Phi \right] _{%
\mathcal{FN}}=0\right\} $.
\end{definition}

\begin{remark}
\label{dfi=somme finie} Let $e_{\Phi }$ the canonical solution of
Maurer-Cartan equation corresponding to $\Phi \in \Lambda ^{1}M\otimes TM$.
Suppose that $e_{\Phi }$ is of finite type $r$. Then 
\begin{equation*}
e_{\Phi }=\sum_{k=1}^{r+1}\gamma _{k}.
\end{equation*}
\end{remark}

\begin{proposition}
\label{Finite type 0}Let $\Phi \in \Lambda ^{1}M\otimes TM$ such that $%
R_{\Phi }$ is invertible. The following are equivalent:

$i)$ The canonical solution $e_{\Phi }$ of Maurer-Cartan equation
coresponding to $\Phi $ is of finite type $0$.

$ii)$ $e_{\Phi }$ is $\daleth $-closed.

$iii)$ $d_{\Phi }$ is $\daleth $-closed.

$iv)$ $N_{\Phi }=0$.
\end{proposition}

\begin{proof}
$i)\iff ii)$ Suppose that the canonical solution $e_{\Phi }$ of
Maurer-Cartan equation coresponding to $\Phi $ is of finite type $0$. Then
by Remark \ref{dfi=somme finie} it follows that

\begin{equation*}
e_{\Phi }=\mathcal{\gamma }_{1}=\mathcal{L}_{\Phi }
\end{equation*}%
and by Lemma \ref{daled ferme} it follows that $e_{\Phi }$ is $\daleth $%
-closed.

Conversely, suppose $\daleth e_{\Phi }=0$. By using again Lemma \ref{daled
ferme} it follows that $e_{\Phi }\in \mathcal{L}\left( M\right) $. In
particular $\mathcal{I}_{b\left( \Phi \right) }=0$, so $\left[ \Phi ,\Phi %
\right] _{\mathcal{FN}}=0$.

$ii)\iff iii)$ We have $d=\mathcal{L}_{Id_{T\left( M\right) }}$, so $\daleth
d=0$. Since $e_{\Phi }=$ $d_{\Phi }-d$ the assertion follows.

$i)\iff iv)$ By Proposition \ref{Calcul [,]FN}, $\left[ \Phi ,\Phi \right] _{%
\mathcal{FN}}=2N_{\Phi }$ so $N_{\Phi }=0$ if and only if $\left[ \Phi ,\Phi %
\right] _{\mathcal{FN}}=0$.
\end{proof}

By using Proposition \ref{Finite type 0} we obtain:

\begin{corollary}
Let $M$ be a smooth manifold and $J$ an almost complex structure on $M$.
Then the canonical solution associated to $J$ is of finite type $0$ if and
only if $N_{J}=0$, i.e. if and only if $J$ is integrable.
\end{corollary}

\begin{theorem}
\label{Type 0 and 1}Let $M$ be a smooth manifold and $\xi \subset TM$ a
distribution. Let $\zeta \subset TM$ such that $TM=\xi \oplus \zeta $ and
consider $\Phi \in End\left( TM\right) $ defined by $\Phi =0$ on $\xi $ and $%
\Phi =Id$ on $\mathbb{\zeta }$. Then:

\begin{enumerate}
\item The canonical solution associated to $\Phi $ is of finite type $0$ if
and only if $\xi $ and $\zeta $ are integrable.

\item The canonical solution associated to $\Phi $ is of finite type $1$ if
and only if $\xi $ is integrable and $\zeta $ is not integrable.

\item If $\xi $ is not integrable then $\Phi ^{k}\left[ \Phi ,\Phi \right] _{%
\mathcal{FN}}\neq 0$ for every $k\in \mathbb{N}$.
\end{enumerate}
\end{theorem}

\begin{proof}
Let $Y,Z\in \xi $. Since $\Phi ^{k}=\Phi $ for every $k\geqslant 1$ and $%
\Phi Y=\Phi Z=0$,%
\begin{eqnarray}
\left( \Phi \left[ \Phi ,\Phi \right] _{\mathcal{FN}}\right) \left(
Y,Z\right) &=&\Phi \left( \left[ \Phi ,\Phi \right] _{\mathcal{FN}}\left(
Y,Z\right) \right)  \notag \\
&=&\Phi \left( \left[ \Phi Y,\Phi Z\right] +\Phi ^{2}\left[ Y,Z\right] -\Phi %
\left[ \Phi Y,Z\right] -\Phi \left[ Y,\Phi Z\right] \right) =\Phi \left( %
\left[ Y,Z\right] \right) .  \label{Fi[Fi,Fi](Y,Z)}
\end{eqnarray}%
Suppose that the canonical solution associated to $\Phi $ is of finite type $%
\leqslant 1$. Then $\Phi \left[ \Phi ,\Phi \right] _{\mathcal{FN}}\left(
Y,Z\right) =0$ for every $Y,Z\in \xi $ and by (\ref{Fi[Fi,Fi](Y,Z)} it
follows that $\left[ Y,Z\right] \in \xi $. Therefore $\xi $ is integrable by
the theorem of Frobenius.

Suppose now that $\xi $ is not integrable. There exist $Y,Z\in \xi $ such
that $\left[ Y,Z\right] \notin \xi $. By (\ref{Fi[Fi,Fi](Y,Z)}) we obtain%
\begin{eqnarray*}
\left( \Phi ^{k}\left[ \Phi ,\Phi \right] _{\mathcal{FN}}\right) \left(
Y,Z\right) &=&\Phi ^{k}\left( \left[ \Phi Y,\Phi Z\right] +\Phi ^{k+2}\left[
Y,Z\right] -\Phi ^{k+1}\left[ \Phi Y,Z\right] -\Phi ^{k+1}\left[ Y,\Phi Z%
\right] \right) \\
&=&\Phi \left( \left[ Y,Z\right] \right) \neq 0
\end{eqnarray*}%
for every $k\geqslant 1$.

Conversely, suppose that $\xi $ is integrable. Then for every $V,W\in \xi $,
we have $\left[ V,W\right] \in \xi $, so $\Phi \left( \left[ V,W\right]
\right) =0$. Since $\Phi \left( TM\right) \subset \mathbb{\zeta }$, for
every $V\in TM$, there exist unique $V_{\xi }\in \xi $ and $V_{\zeta }\in
\zeta $ such that $V=V_{\xi }+V_{\zeta }$, and $\Phi V=V_{\zeta }$. Since $%
\Phi ^{k}=\Phi $ for every $k\geqslant 1$,%
\begin{eqnarray*}
\left( \Phi \left[ \Phi ,\Phi \right] _{\mathcal{FN}}\right) \left(
V,W\right) &=&\Phi \left( \left[ \Phi ,\Phi \right] _{\mathcal{FN}}\left(
V,W\right) \right) \\
&=&\Phi \left( \left[ \Phi V,\Phi W\right] +\Phi ^{2}\left[ V,W\right] -\Phi %
\left[ \Phi V,W\right] -\Phi \left[ V,\Phi W\right] \right) \\
&=&\Phi \left( \left[ V_{\zeta },W_{\zeta }\right] \right) +\Phi \left( %
\left[ V_{\xi }+V_{\zeta },W_{\xi }+W_{\zeta }\right] \right) \\
&&-\Phi \left( \left[ V_{\zeta },W_{\xi }+W_{\zeta }\right] \right) -\Phi
\left( \left[ V_{\xi }+V_{\zeta },W_{\zeta }\right] \right) \\
&=&\Phi \left[ V_{\zeta },W_{\zeta }\right] +\Phi \left[ V_{\zeta },W_{\xi }%
\right] +\Phi \left[ V_{\xi },W_{\zeta }\right] +\Phi \left[ V_{\zeta
},W_{\zeta }\right] \\
&&-\Phi \left[ V_{\zeta },W_{\xi }\right] -\Phi \left[ V_{\zeta },W_{\zeta }%
\right] -\Phi \left[ V_{\xi },W_{\zeta }\right] -\Phi \left[ V_{\zeta
},W_{\zeta }\right] \\
&=&0
\end{eqnarray*}%
and it follows that the canonical solution associated to $\Phi $ is of
finite type $\leqslant 1$.

If $\zeta $ is integrable too, $\Phi \left[ V_{\zeta },W_{\zeta }\right] =%
\left[ V_{\zeta },W_{\zeta }\right] $ for every $V,W\in TM$ and%
\begin{eqnarray*}
\left( \left[ \Phi ,\Phi \right] _{\mathcal{FN}}\right) \left( V,W\right) &=&%
\left[ \Phi V,\Phi W\right] +\Phi ^{2}\left[ V,W\right] -\Phi \left[ \Phi V,W%
\right] -\Phi \left[ V,\Phi W\right] \\
&=&\left[ V_{\zeta },W_{\zeta }\right] +\Phi \left[ V_{\xi }+V_{\zeta
},W_{\xi }+W_{\zeta }\right] \\
&&-\Phi \left[ V_{\zeta },W_{\xi }+W_{\zeta }\right] -\Phi \left[ V_{\xi
}+V_{\zeta },W_{\zeta }\right] \\
&=&\left[ V_{\zeta },W_{\zeta }\right] +\Phi \left[ V_{\zeta },W_{\xi }%
\right] +\Phi \left[ V_{\xi },W_{\zeta }\right] -\Phi \left[ V_{\zeta
},W_{\xi }\right] -\Phi \left[ V_{\zeta },W_{\zeta }\right] =0,
\end{eqnarray*}%
so the canonical solution associated to $\Phi $ is of finite type $0$.

If $\xi $ is integrable and $\zeta $ is not integrable, there exists $Y,Z\in
\zeta $ such that $\left[ Y,Z\right] \notin \zeta $, so 
\begin{eqnarray*}
\left( \left[ \Phi ,\Phi \right] _{\mathcal{FN}}\right) \left[ Y,Z\right] &=&%
\left[ \Phi Y,\Phi Z\right] +\Phi ^{2}\left[ Y,Z\right] -\Phi \left[ \Phi Y,Z%
\right] -\Phi \left[ Y,\Phi Z\right] \\
&=&\left[ Y,Z\right] +\Phi \left[ Y,Z\right] -\Phi \left[ Y,Z\right] -\Phi %
\left[ Y,Z\right] =\left[ Y,Z\right] -\Phi \left[ Y,Z\right] \neq 0
\end{eqnarray*}%
and the Theorem follows.
\end{proof}

\begin{corollary}
\label{Type 0}Let $M$ be a smooth manifold and $\xi \subset TM$ a
co-orientable distribution of codimension $1$. There exist $X\in \mathfrak{X}%
\left( M\right) $ and $\gamma \in \Lambda ^{1}\left( M\right) $ such $\xi
=\ker \gamma $ and $\iota _{X}\gamma =1$. We have $T\left( M\right) =\xi
\oplus \mathbb{R}\left[ X\right] $ and we consider $\Phi \in End\left(
TM\right) $ defined by $\Phi =0$ on $\xi $ and $\Phi =Id$ on $\mathbb{R}%
\left[ X\right] $. Then the canonical solution associated to $\Phi $ is of
finite type $0$ if and only if $\xi $ is integrable.
\end{corollary}

\begin{proof}
We apply Proposition \ref{Type 0 and 1} for $\eta =\mathbb{R}\left[ X\right] 
$ which is obviously integrable.
\end{proof}

\begin{theorem}
\label{Type r>1}Let $M$ be a smooth manifold of dimension $n$ and $\xi ,\tau
\subset TM$ distributions such that $\xi \subsetneqq \tau $. We consider $%
\eta ,\zeta \subset TM$ distributions such that $\tau =\xi \oplus \eta $ and 
$TM=\tau \oplus \zeta $ and let $A:\eta \rightarrow \xi $ , $B:\eta
\rightarrow \eta $ such that $\xi =\ker K$, where $K:\tau \rightarrow \tau $
is defined by $K=0$ on $\xi $ and $K=A+B$ on $\eta $. We suppose that there
exists a natural number $m\geqslant 1$ such that $K^{m}=0$. Let $\Phi \in
End\left( TM\right) $ defined by $\Phi =K$ on $\tau $ and $\Phi =Id$ on $%
\zeta $. The following are equivalent:

\begin{enumerate}
\item $\tau $ is integrable.

\item The canonical solution associated to $\Phi $ is of finite type $%
\leqslant m$.
\end{enumerate}
\end{theorem}

\begin{proof}
We have%
\begin{equation*}
K=\left( 
\begin{array}{ccc}
& \underbrace{\dim \xi } & \underbrace{\dim \eta } \\ 
\dim \xi \ \{ & 0 & A \\ 
\dim \eta \ \{ & 0 & B%
\end{array}%
\right)
\end{equation*}%
and%
\begin{equation*}
\Phi =\left( 
\begin{array}{ccc}
& \underbrace{\dim \tau } & \underbrace{\dim \zeta } \\ 
\dim \tau \ \{ & K & 0 \\ 
\dim \zeta \ \{ & 0 & Id%
\end{array}%
\right) ,
\end{equation*}%
So $K^{m}=0$, $\Phi ^{m}=0$ on $\tau $ and $\Phi ^{m}=Id$ on $\zeta $.

Let $Y,Z\in TM$, $Y=Y_{\tau }+Y_{\zeta }$, $Y_{\tau }=Y_{\xi }+Y_{\eta }$, $%
Z=Z_{\tau }+Z_{\zeta }$,\ $Z_{\tau }=Z_{\xi }+Z_{\eta }$, $Y_{\xi },Z_{\xi
}\in \xi $, $Y_{\eta },Z_{\eta }\in \eta $, $Y_{\zeta },Z_{\zeta }\in \zeta $%
. We have $\Phi ^{m+j}=\Phi ^{m}$ for every $j\in \mathbb{N}$, so%
\begin{eqnarray}
\left( \Phi ^{m}\left[ \Phi ,\Phi \right] _{\mathcal{FN}}\right) \left[ Y,Z%
\right] &=&\Phi ^{m}\left( \left[ \Phi Y,\Phi Z\right] +\Phi ^{2}\left[ Y,Z%
\right] -\Phi \left[ \Phi Y,Z\right] -\Phi \left[ Y,\Phi Z\right] \right) 
\notag \\
&=&\Phi ^{m}\left[ \Phi Y,\Phi Z\right] +\Phi ^{m}\left[ Y,Z\right] -\Phi
^{m}\left[ \Phi Y,Z\right] -\Phi ^{m}\left[ Y,\Phi Z\right] .  \label{e}
\end{eqnarray}

Suppose that $\tau $ is integrable. Since $\xi =\ker \Phi $, 
\begin{eqnarray*}
\Phi Y &=&\Phi Y_{\eta }+\Phi Y_{\zeta }=AY_{\eta }+BY_{\eta }+Y_{\zeta
}=C_{\tau }+Y_{\zeta }, \\
\Phi Z &=&\Phi Z_{\eta }+\Phi Z_{\zeta }=AZ_{\eta }+BZ_{\eta }+Z_{\zeta
}=D_{\tau }+Y_{\zeta },
\end{eqnarray*}%
where $C_{\tau }=AY_{\eta }+BY_{\eta }\in \tau $ and $D_{\tau }=AZ_{\eta
}+BZ_{\eta }\in \tau $. Since $\tau $ is integrable, $\left[ C_{\tau
},D_{\tau }\right] \in \tau =Ker$~$\Phi ^{m}$, so $\Phi ^{m}\left( \left[
C_{\tau },D_{\tau }\right] \right) =0$ and 
\begin{equation}
\Phi ^{m}\left[ \Phi Y,\Phi Z\right] =\Phi ^{m}\left( \left[ C_{\tau
},Z_{\zeta }\right] \right) +\Phi ^{m}\left( \left[ Y_{\zeta },D_{\tau }%
\right] \right) +\Phi ^{m}\left( \left[ Y_{\zeta },Z_{\zeta }\right] \right)
.  \label{a}
\end{equation}

Similarly, since $\left[ Y_{\tau },Z_{\tau }\right] ,\ \left[ C_{\tau
},Z_{\tau }\right] ,\ \left[ Y_{\tau },D_{\tau }\right] \in \tau =Ker$~$\Phi
^{m}$ 
\begin{equation}
\Phi ^{m}\left[ Y,Z\right] =\Phi ^{m}\left( \left[ Y_{\tau }+Y_{\zeta
},Z_{\tau }+Z_{\zeta }\right] \right) =\Phi ^{m}\left( \left[ Y_{\zeta
},Z_{\tau }\right] \right) +\Phi ^{m}\left( \left[ Y_{\tau },Z_{\zeta }%
\right] \right) +\Phi ^{m}\left( \left[ Y_{\zeta },Z_{\zeta }\right] \right)
,  \label{b}
\end{equation}

\begin{equation}
\Phi ^{m}\left[ \Phi Y,Z\right] =\Phi ^{m}\left[ C_{\tau }+Y_{\zeta
},Z_{\tau }+Z_{\zeta }\right] =\Phi ^{m}\left[ C_{\tau },Z_{\zeta }\right]
+\Phi ^{m}\left[ Y_{\zeta },Z_{\tau }\right] +\Phi ^{m}\left[ Y_{\zeta
},Z_{\zeta }\right] .  \label{c}
\end{equation}

\begin{equation}
\Phi ^{m}\left[ Y,\Phi Z\right] =\Phi ^{m}\left[ Y_{\tau }+Y_{\zeta
},D_{\tau }+Z_{\zeta }\right] =\Phi ^{m}\left[ Y_{\zeta },D_{\tau }\right]
+\Phi ^{m}\left[ Y_{\tau },Z_{\zeta }\right] +\Phi ^{m}\left[ Y_{\zeta
},Z_{\zeta }\right]  \label{d}
\end{equation}%
Replacing \ (\ref{a}), (\ref{b}), (\ref{c}), (\ref{d}) in (\ref{e}) we obtain%
\begin{eqnarray*}
\left( \Phi ^{m}\left[ \Phi ,\Phi \right] _{\mathcal{FN}}\right) \left[ Y,Z%
\right] &=&\Phi ^{m}\left( \left[ C_{\tau },Z_{\zeta }\right] \right) +\Phi
^{m}\left( \left[ Y_{\zeta },D_{\tau }\right] \right) +\Phi ^{m}\left( \left[
Y_{\zeta },Z_{\zeta }\right] \right) \\
&&+\Phi ^{m}\left( \left[ Y_{\zeta },Z_{\tau }\right] \right) +\Phi
^{m}\left( \left[ Y_{\tau },Z_{\zeta }\right] \right) +\Phi ^{m}\left( \left[
Y_{\zeta },Z_{\zeta }\right] \right) \\
&&-\left( \Phi ^{m}\left[ C_{\tau },Z_{\zeta }\right] +\Phi ^{m}\left[
Y_{\zeta },Z_{\tau }\right] +\Phi ^{m}\left[ Y_{\zeta },Z_{\zeta }\right]
\right) \\
&&-\left( \Phi ^{m}\left[ Y_{\zeta },D_{\tau }\right] +\Phi ^{m}\left[
Y_{\tau },Z_{\zeta }\right] +\Phi ^{m}\left[ Y_{\zeta },Z_{\zeta }\right]
\right) \\
&=&0,
\end{eqnarray*}%
and it follows that the canonical solution associated to $\Phi $ is of
finite type $\leqslant m$.

Conversely, suppose that the canonical solution of the Maurer-Cartan
equation associated to $\Phi $ is of finite type $k\leqslant m$, i.e. $\Phi
^{m}\left[ \Phi ,\Phi \right] _{\mathcal{FN}}=0$. We will prove that $\left[
Y,Z\right] \in \tau =Ker$~$\Phi ^{m}$ for every $Y,Z\in \tau $ by taking in
account several cases.

a) Let $Y,Z\in \xi $. Then $\Phi Y=\Phi Z=0$ and by using (\ref{e}) we obtain%
\begin{equation*}
\Phi ^{m}\left[ \Phi ,\Phi \right] _{\mathcal{FN}}\left( \left[ Y,Z\right]
\right) =\Phi ^{m+2}\left[ Y,Z\right] =\Phi ^{m}\left[ Y,Z\right] =0.
\end{equation*}%
So $\left[ Y,Z\right] \in \tau $.

b) Let $Y\in \xi ,Z\in \tau $, $Z=Z_{\xi }+Z_{\eta }$. Then%
\begin{equation}
\left[ Y,Z\right] =\left[ Y,Z_{\xi }\right] +\left[ Y,Z_{\eta }\right] .
\label{g}
\end{equation}%
By a) $\left[ Y,Z_{\xi }\right] \in \tau $ and from (\ref{g}) it follows
that $\left[ Y,Z\right] \in \tau $ if and only if $\left[ Y,Z_{\eta }\right]
\in \tau $.

Since $\Phi Y=0$, by using (\ref{e}) we have

\begin{eqnarray*}
\left[ \Phi ,\Phi \right] _{\mathcal{FN}}\left( \left[ Y,Z_{\eta }\right]
\right) &=&\left[ \Phi Y,\Phi Z_{\eta }\right] +\Phi ^{2}\left[ Y,Z_{\eta }%
\right] -\Phi \left[ \Phi Y,Z_{\eta }\right] -\Phi \left[ Y,\Phi Z_{\eta }%
\right] \\
&=&\Phi ^{2}\left[ Y,Z_{\eta }\right] -\Phi \left[ Y,\Phi Z_{\eta }\right] ,
\end{eqnarray*}%
and%
\begin{eqnarray*}
\Phi ^{m}\left[ \Phi ,\Phi \right] _{\mathcal{FN}}\left( \left[ Y,Z_{\eta }%
\right] \right) &=&\Phi ^{m}\left[ Y,Z_{\eta }\right] -\Phi ^{m}\left[
Y,\Phi Z_{\eta }\right] \\
&=&\Phi ^{m}\left( \left[ Y,\left( Id-\Phi \right) Z_{\eta }\right] \right)
=0.
\end{eqnarray*}%
In particular $\left[ Y,\left( Id-\Phi \right) Z_{\eta }\right] \in \tau $
for every $Z_{\eta }\in \eta $.

But 
\begin{equation}
\det \left( Id-\Phi _{\left\vert \eta \right. }\right) =\det \left(
Id-B\right) =1,  \label{f}
\end{equation}%
so for every $X_{\eta }\in \eta $ there exists $Z_{\eta }\in \eta $ such
that $X_{\eta }=\left( Id-\Phi \right) Z_{\eta }$. It follows that $\left[
Y,Z\right] \in \tau $ for every $Y\in \xi $ and $Z\in \tau $.

c) Let $Y,Z\in \tau $, $Y=Y_{\xi }+Y_{\eta }$, $Z=Z_{\xi }+Z_{\eta }$, $%
Y_{\xi },Z_{\xi }\in \xi $ $Y_{\eta },Z_{\eta }\in \eta $. Since 
\begin{equation*}
\left[ Y,Z\right] =\left[ Y_{\xi }+Y_{\eta },Z_{\xi }+Z_{\eta }\right] =%
\left[ Y_{\xi },Z_{\xi }\right] +\left[ Y_{\eta },Z_{\xi }\right] +\left[
Y_{\xi },Z_{\eta }\right] +\left[ Y_{\eta },Z_{\eta }\right] .
\end{equation*}%
and $\left[ Y_{\xi },Z_{\xi }\right] ,\left[ Y_{\eta },Z_{\xi }\right] ,%
\left[ Y_{\xi },Z_{\eta }\right] \in \tau $ it follows that $\left[ Y,Z%
\right] \in \tau $ if and only if $\left[ Y_{\eta },Z_{\eta }\right] \in
\tau $.

We have 
\begin{eqnarray*}
\left[ \Phi ,\Phi \right] _{\mathcal{FN}}\left( \left[ Y_{\eta },Z_{\eta }%
\right] \right) &=&\left[ \Phi Y_{\eta },\Phi Z_{\eta }\right] +\Phi ^{2}%
\left[ Y_{\eta },Z_{\eta }\right] -\Phi \left[ \Phi Y_{\eta },Z_{\eta }%
\right] -\Phi \left[ Y_{\eta },\Phi Z_{\eta }\right] \\
&=&\left[ AY_{\eta }+BY_{\eta },AZ_{\eta }+BZ_{\eta }\right] +\Phi ^{2}\left[
Y_{\eta },Z_{\eta }\right] \\
&&-\Phi \left[ AY_{\eta }+BY_{\eta },Z_{\eta }\right] -\Phi \left[ Y_{\eta
},AZ_{\eta }+BZ_{\eta }\right] .
\end{eqnarray*}%
Since $AY_{\eta },AZ_{\eta }\in \xi $ it follows that $\left[ AY_{\eta
},AZ_{\eta }\right] ,\left[ AY_{\eta },BZ_{\eta }\right] ,\left[ BY_{\eta
},AZ_{\eta }\right] ,\left[ AY_{\eta },Z_{\eta }\right] ,\left[ Y_{\eta
},AZ_{\eta }\right] \in \tau =\ker \Phi ^{r}$ and%
\begin{eqnarray*}
\Phi ^{m}\left[ \Phi ,\Phi \right] _{\mathcal{FN}}\left( \left[ Y_{\eta
},Z_{\eta }\right] \right) &=&\Phi ^{m}\left[ BY_{\eta },BZ_{\eta }\right]
+\Phi ^{m}\left[ Y_{\eta },Z_{\eta }\right] -\Phi ^{m}\left[ BY_{\eta
},Z_{\eta }\right] -\Phi ^{m}\left[ Y_{\eta },BZ_{\eta }\right] \\
&=&\Phi ^{m}\left[ Y_{\eta },\left( Id-B\right) Z_{\eta }\right] -\Phi ^{m}%
\left[ BY_{\eta },\left( Id-B\right) Z_{\eta }\right] \\
&=&\Phi ^{m}\left[ \left( Id-B\right) Y_{\eta },\left( Id-B\right) Z_{\eta }%
\right] =0
\end{eqnarray*}

for every $Y_{\eta },Z_{\eta }\in \eta $.

As before, by (\ref{f}) it follows that $\Phi ^{r}\left[ Y_{\eta },Z_{\eta }%
\right] =0$ for every $Y_{\eta },Z_{\eta }\in \eta $ and this implies that $%
\left[ Y,Z\right] \in \tau $ for every $Y,Z\in \tau $.
\end{proof}

In order to compute the type of the canonical solution of Theorem \ref{Type
r>1} we need the following elementary lemma:

\begin{lemma}
\label{min nilpotent}Let%
\begin{equation*}
K=\left( 
\begin{array}{ccc}
& \underbrace{s} & \underbrace{d-s} \\ 
s\ \{ & 0 & A \\ 
d-s\ \{ & 0 & B%
\end{array}%
\right)
\end{equation*}%
a $\left( d,d\right) $ nilpotent matrix of rank $d-s>0$, $s\geqslant 1$. Set 
$r=\min \left\{ m\in \mathbb{N}:\ K^{m}=0\right\} $. Then $r=\min \left\{
m\in \mathbb{N}:\ m\geqslant \frac{d}{s}\right\} $.
\end{lemma}

\begin{proof}
Since $K$ is nilpotent of maximal rank we may suppose that 
\begin{equation*}
K=\left( 
\begin{array}{ccc}
& \underbrace{s} & \underbrace{d-s} \\ 
s\ \{ & 0 & Id \\ 
d-s\ \{ & 0 & 0%
\end{array}%
\right) .
\end{equation*}%
By induction it follows that if $d-js>0$, we have%
\begin{equation*}
K^{j}=\left( 
\begin{array}{ccc}
& \underbrace{js} & \underbrace{d-js} \\ 
js\ \{ & 0 & Id \\ 
d-js\ \{ & 0 & 0%
\end{array}%
\right) \neq 0
\end{equation*}%
and $K^{j}=0$ for each $j\in \mathbb{N}^{\ast }$ such that $d-js\leqslant 0$.
\end{proof}

\begin{notation}
Let $\xi \subset TM$ a distribution. We denote by $\xi ^{\ast }$ the
smallest involutive subset of $TM$ such that $\xi \subset \xi ^{\ast }$. If $%
\mathcal{E}=\left\{ X_{1},\cdot \cdot \cdot ,X_{s}\right\} $ are generators
of $\xi $ on an open subset $U$ of $M$, then for every $x\in U$, $\xi
_{x}^{\ast }$ is the linear subspace of $T_{x}M$ generated by $\left[
X_{i_{1}},\left[ X_{i_{2}},\left[ \cdot \cdot \cdot ,X_{i_{k}}\right] \right]
\right] \left( x\right) $, $k\geqslant 1,\ 1\leqslant i_{k}\leqslant s.$
\end{notation}

\begin{remark}
If $\dim \xi _{x}^{\ast }$ is independent of $x$, $\xi ^{\ast }$ is a
distribution, but in general $\dim \xi _{x}^{\ast }$ depends on $x$. If $\xi
^{\ast }$ is a distribution, then $\xi ^{\ast }$ is the smallest integrable
distribution containing $\xi $ \cite{Sussmann1973}.
\end{remark}

\begin{corollary}
\label{Csi *}Let $M$ be a smooth manifold of dimension $n$, $\xi \subset TM$
a distribution of dimension $s$ such that $\xi ^{\ast }$ is a distribution
of dimension $d$. Then for every $x\in M$ there exists a neighborhood $U$ of 
$x$ and $\Phi \in \Lambda ^{1}\left( U,TU\right) $ such that the canonical
solution of the Maurer-Cartan equation associated to $\Phi $ is of finite
type $\leqslant r$, where $r=\min \left\{ m\in \mathbb{N}:\ m\geqslant \frac{%
d}{s}\right\} $.
\end{corollary}

\begin{proof}
If $\xi $ is integrable, $d=s$, $r=1$ and the corollary follows from Theorem %
\ref{Type 0 and 1}.

Suppose that $\xi $ is not integrable, i. e. $d>s$. For each $x\in M$ there
exists a neighborhood $U$ of $x$ and a basis $\left( X_{1},\cdot \cdot \cdot
,X_{n}\right) $ of $TM$ on $U$ such that $\left( X_{1},\cdot \cdot \cdot
,X_{s}\right) $ is a basis of $\xi $ and $\left( X_{1},\cdot \cdot \cdot
,X_{d}\right) $ is a basis of $\xi ^{\ast }$ on $U$.

We define $\Phi \in End\left( TU\right) $ as $\Phi X_{i}=0$, $i=1,\cdot
\cdot \cdot ,s$ $\Phi \left( X_{i}\right) =X_{i-s}$, $i=s+1,\cdot \cdot
\cdot ,d$, $\Phi \left( X_{i}\right) =X_{i}$, $i=d+1,\cdot \cdot \cdot ,n$.
Then the matrix of $\Phi $ in the basis $\left( X_{1},\cdot \cdot \cdot
,X_{n}\right) $ is 
\begin{equation*}
\Phi =\left( 
\begin{array}{ccc}
& \underbrace{d} & \underbrace{n-d} \\ 
d\ \{ & K & 0 \\ 
n-d\ \{ & 0 & Id%
\end{array}%
\right)
\end{equation*}%
where 
\begin{equation*}
K=\left( 
\begin{array}{ccc}
& \underbrace{s} & \underbrace{d-s} \\ 
s\ \{ & 0 & Id \\ 
d-s\ \{ & 0 & 0%
\end{array}%
\right) .
\end{equation*}

Since $r\geqslant 2$, by Lemma \ref{min nilpotent} and Theorem \ref{Type r>1}%
, the canonical equation solution of the Maurer-Cartan equation associated
to $\Phi $ is of finite type $\leqslant r$.
\end{proof}

\begin{remark}
In \cite{PdBandAI2017} it is proved that the deformation theory in the DGLA $%
\left( \mathcal{D}^{\ast }\left( M\right) ,\daleth ,\left[ \cdot ,\cdot %
\right] \right) $ is not obstructed but it is level-wise obstructed.
\end{remark}

\section{Deformations of foliations of codimension $1$}

\begin{definition}
By a differentiable family of deformations of an integrable distribution $%
\xi $ we mean a differentiable family $\omega :\mathcal{D}=\left( \xi
_{t}\right) _{t\in I}\mapsto t\in I=]-a,a[$, $a>0$, of integrable
distributions such that $\xi _{0}=\omega ^{-1}\left( 0\right) =\xi $.
\end{definition}

\begin{remark}
An integrable distribution $\xi $ of codimension $1$ in a smooth manifold $L$
is called co-orientable if the normal space to the foliation defined by $\xi 
$ is orientable. We recall that $\xi $ is co-orientable if and only if there
exists a $1$-form $\gamma $ on $L$ such that $\xi =\ker ~\gamma $ (see for
ex. \cite{Godbillon91}). A couple $\left( \gamma ,X\right) $ where $\gamma
\in \wedge ^{1}\left( L\right) $ and $X$ is a vector field on $L$ such that $%
\ker ~\gamma =\xi $ and $\gamma \left( X\right) =1$ was called a DGLA
defining couple in \cite{PdbAI15}.

If $\left( \xi _{t}\right) _{t\in I}$ is a differentiable family of
deformations of an integrable co-orientable distribution $\xi $, then the
distribution $\xi _{t}$ is co-orientable for $t$ small enough. So, if $\xi $
is an integrable co-orientable distribution of codimension $1$ in $L$ and $%
\left( \xi _{t}\right) _{t\in I}$ is a differentiable family of deformations
of $\xi $ we may consider a DGLA defining couple $\left( \gamma
_{t},X_{t}\right) $ for every $t$ small enough such that $t\mapsto \left(
\gamma _{t},X_{t}\right) $ is differentiable on a neighborhood of the origin.
\end{remark}

\begin{lemma}
\label{gama tens X}Let $L$ be a $C^{\infty }$ manifold and $\xi \subset
T\left( L\right) $ a co-orientable distribution of codimension $1$. Let $%
\left( \gamma ,X\right) $ be a DGLA defining couple and denote $\Phi \in
End\left( TM\right) $ the endomorphism corresponding to $\gamma \otimes X\in
\Lambda ^{1}M\otimes TM$. Then $\Phi $ is defined on $TM=\xi \oplus \mathbb{R%
}\left[ X\right] $ as $\Phi =0$ on $\xi $ and $\Phi =Id$ on $\mathbb{R}\left[
X\right] $.
\end{lemma}

\begin{proof}
Let $Y=Y_{\xi }+\lambda X$ vector fields on $L$, $V_{\xi }\in \xi $, $%
\lambda \in \mathbb{R}$. Then 
\begin{equation*}
\left( \gamma \otimes X\right) \left( Y\right) =\gamma \left( Y\right)
X=\gamma \left( Y_{\xi }+\lambda X\right) X=\lambda X.
\end{equation*}
\end{proof}

\begin{lemma}
\label{Frobenius} Let $L$ be a $C^{\infty }$ manifold and $\xi \subset
T\left( L\right) $ a co-orientable distribution of codimension $1$. Let $%
\left( \gamma ,X\right) $ be a DGLA defining couple. Then the following are
equivalent:

i) $\xi$ is integrable;

ii) $d\gamma =-\iota _{X}d\gamma \wedge \gamma $.

iii) $\left[ \gamma \otimes X,\gamma \otimes X\right] _{\mathcal{FN}}=0$.
\end{lemma}

\begin{proof}
$i)\iff ii)$ is a variant of the theorem of Frobenius and it was proved in 
\cite{PdbAI15}.

$ii)\iff iii)$. We have 
\begin{eqnarray*}
\left[ \gamma \otimes X,\gamma \otimes X\right] _{\mathcal{FN}} &=&\gamma
\wedge \mathcal{L}_{X}\gamma \otimes X-\mathcal{L}_{X}\gamma \wedge \gamma
\otimes X-\left( d\gamma \wedge \iota _{X}\gamma \otimes X+\iota _{X}\gamma
\wedge d\gamma \otimes X\right) \\
&=&2\gamma \wedge \mathcal{L}_{X}\gamma \otimes X-2d\gamma \otimes X=2\left(
\gamma \wedge d\iota _{X}\gamma +\gamma \wedge \iota _{X}d\gamma -d\gamma
\right) \otimes X \\
&=&2\left( \gamma \wedge \iota _{X}d\gamma -d\gamma \right) \otimes X.
\end{eqnarray*}
\end{proof}

We recall the following lemma from \cite{PdbAI15}:

\begin{lemma}
\label{Forms=DGLA}Let $L$ be a $C^{\infty }$ manifold and $X$ a vector field
on $L$. For $\alpha ,\beta \in \Lambda ^{\ast }\left( L\right) $, set%
\begin{equation}
\left\{ \alpha ,\beta \right\} =\mathcal{L}_{X}\alpha \wedge \beta -\alpha
\wedge \mathcal{L}_{X}\beta  \label{Lie braket 1}
\end{equation}%
where $\mathcal{L}_{X}$ is the Lie derivative. Then $\left( \Lambda ^{\ast
}\left( L\right) ,d,\left\{ \cdot ,\cdot \right\} \right) $ is a DGLA.
\end{lemma}

\begin{proposition}
\label{delta alfa}Let $L$ be a $C^{2}$ manifold and $\xi \subset T\left(
L\right) $ an integrable co-orientable distribution of codimension $1$. Let $%
\left( \xi _{t}\right) _{t\in I}$ be a differentiable family of deformations
of $\xi $ such that $\xi _{t}$ is co-orientable and integrable for every $%
t\in I$ and let $\left( \gamma _{t},X_{t}\right) $ a DGLA defining couple
for $\xi _{t}$ such that $t\mapsto \left( \gamma _{t},X_{t}\right) $ is
differentiable on $I$. Denote $\gamma =\gamma _{0}$, $\alpha =\frac{d\gamma
_{t}}{dt}_{\left\vert t=0\right. }$, $X=X_{0}$,\ $Y=\frac{dX_{t}}{dt}%
_{\left\vert t=0\right. }$. Then%
\begin{equation*}
\delta \alpha +\mathcal{L}_{Y}\gamma \wedge \gamma =0
\end{equation*}%
where 
\begin{equation*}
\delta =d+\left\{ \gamma ,\cdot \right\}
\end{equation*}%
and $\left\{ \cdot ,\cdot \right\} $ is defined in (\ref{Lie braket 1}).

In particular $\delta \alpha \left( V,W\right) =0$ for every vector fields $%
V,W$ tangent to $\xi $.
\end{proposition}

\begin{proof}
Since 
\begin{equation*}
\gamma _{t}\left( X_{t}\right) =\left( \gamma +t\alpha +o\left( t\right)
\right) \left( X+tY+o\left( t\right) \right) =1+t\left( \alpha \left(
X\right) +\gamma \left( Y\right) \right) +o\left( t\right) =1
\end{equation*}%
it follows that%
\begin{equation}
\alpha \left( X\right) +\gamma \left( Y\right) =0.  \label{alfa(X)+gama(Y)}
\end{equation}%
Denote $\sigma \left( t\right) =\gamma _{t}\otimes X_{t}\in \Lambda
^{1}M\otimes TM$. By Corollary \ref{Type 0} and Lemma \ref{gama tens X} the
canonical solution of the Maurer-Cartan equation in $\left( \mathcal{D}%
^{\ast }\left( L\right) ,\left[ \cdot ,\cdot \right] ,\daleth \right) $
associated to $\sigma \left( t\right) $ is of finite type $0$ for each $t$,
so $\left[ \sigma \left( t\right) ,\sigma \left( t\right) \right] _{\mathcal{%
FN}}=0$ for every $t$. We have%
\begin{eqnarray*}
\sigma \left( t\right) &=&\gamma _{t}\otimes X_{t}=\left( \gamma +t\alpha
+o\left( t\right) \right) \otimes \left( X+tY+o\left( t\right) \right) \\
&=&\gamma \otimes X+t\left( \alpha \otimes X+\gamma \otimes Y\right)
+o\left( t\right)
\end{eqnarray*}%
and%
\begin{equation*}
\left[ \sigma \left( t\right) ,\sigma \left( t\right) \right] _{\mathcal{FN}%
}=\left[ \gamma \otimes X,\gamma \otimes X\right] _{\mathcal{FN}}+2t\left( %
\left[ \gamma \otimes X,\alpha \otimes X+\gamma \otimes Y\right] _{\mathcal{%
FN}}\right) +o\left( t\right) .
\end{equation*}%
By Lemma \ref{Frobenius}, $\left[ \gamma \otimes X,\gamma \otimes X\right] _{%
\mathcal{FN}}=0$, so%
\begin{equation*}
\left[ \sigma \left( t\right) ,\sigma \left( t\right) \right] _{\mathcal{FN}%
}=2t\left( \left[ \gamma \otimes X,\alpha \otimes X+\gamma \otimes Y\right]
_{\mathcal{FN}}\right) +o\left( t\right) =0
\end{equation*}%
and it follows that%
\begin{equation}
\left[ \gamma \otimes X,\alpha \otimes X+\gamma \otimes Y\right] _{\mathcal{%
FN}}=0.  \label{[,]0}
\end{equation}%
But Proposition \ref{Calcul [,]FN} gives%
\begin{eqnarray}
\left[ \gamma \otimes X,\alpha \otimes X\right] _{\mathcal{FN}} &=&\gamma
\wedge \mathcal{L}_{X}\alpha \otimes X-\mathcal{L}_{X}\gamma \wedge \alpha
\otimes X-\left( d\gamma \wedge \iota _{X}\alpha \otimes X+\iota _{X}\gamma
\wedge d\alpha \otimes X\right)  \notag \\
&=&-\left\{ \gamma ,\alpha \right\} \otimes X-\alpha \left( X\right) d\gamma
\otimes X-d\alpha \otimes X  \notag \\
&=&-\delta \alpha \otimes X-\alpha \left( X\right) d\gamma \otimes X
\label{[,]1}
\end{eqnarray}%
and%
\begin{eqnarray*}
\left[ \gamma \otimes X,\gamma \otimes Y\right] _{\mathcal{FN}} &=&\gamma
\wedge \mathcal{L}_{X}\gamma \otimes Y-\mathcal{L}_{Y}\gamma \wedge \gamma
\otimes X \\
&&-d\gamma \wedge \iota _{X}\gamma \otimes Y-\iota _{Y}\gamma \wedge d\gamma
\otimes X \\
&=&\gamma \wedge \iota _{X}d\gamma \otimes Y-\mathcal{L}_{Y}\gamma \wedge
\gamma \otimes X \\
&&-d\gamma \otimes Y-\gamma \left( Y\right) d\gamma \otimes X.
\end{eqnarray*}%
By using Lemma \ref{Frobenius} it follows that 
\begin{equation}
\left[ \gamma \otimes X,\gamma \otimes Y\right] _{\mathcal{FN}}=-\mathcal{L}%
_{Y}\gamma \wedge \gamma \otimes X-\gamma \left( Y\right) d\gamma \otimes X
\label{[,]2}
\end{equation}%
and by (\ref{[,]0}), (\ref{[,]1}) (\ref{[,]2} and (\ref{alfa(X)+gama(Y)}) we
obtain 
\begin{equation*}
-\delta \alpha -\left( \alpha \left( X\right) +\gamma \left( Y\right)
\right) d\gamma -\mathcal{L}_{Y}\gamma \wedge \gamma =-\delta \alpha -%
\mathcal{L}_{Y}\gamma \wedge \gamma =0.
\end{equation*}
\end{proof}

\begin{remark}
A smooth hypersurface in a complex manifold is Levi flat if it admits a
foliation of codimension $1$ by complex manifolds. In \cite{PdbAI15} the
authors studied the deformations of Levi flat hypersurfaces and obtained a
second order elliptic differential equation for the infinitesimal
deformations, which was used to prove the non existence of of transversally
parallelizable Levi flat hypersurfaces in the complex projective plane. In 
\cite{Iordan2018} it is proved that the results of this paragraph lead to
the same second order elliptic differential equation for the infinitesimal
deformations of Levi flat hypersurfaces.
\end{remark}

\renewcommand\baselinestretch{1} 
\bibliographystyle{amsplain}
\bibliography{defintr,deformations,ref1}

\end{document}